\documentclass[12pt]{article}
\usepackage{amssymb}
\usepackage{amsmath}
\usepackage{graphs}
\usepackage{amssymb}
\usepackage{amsthm}
\usepackage{url}

\theoremstyle{plain}
\newtheorem{definition}{Definition}

\newtheorem{theorem}[definition]{Theorem}

\newtheorem{problem}[definition]{Problem}

\newcommand{\lef}{\mathcal{L}}
\newcommand{\pre}{\mathcal{P}}
\newcommand{\n}{\mathcal{N}}
\newcommand{\ri}{\mathcal{R}}

\begin{document}

\title{Mis\`ere Hackenbush is NP-Hard}
\author{Fraser Stewart}
\date{}
\maketitle{}

\begin{abstract}
Hackenbush is a two player game, played on a graph with
coloured edges where players take it in turns to remove edges of
their own colour.  It has been shown that under normal play rules
Red-Blue Hackenbush (all edges are coloured either red or blue) is
NP-hard.  We will show that Red-Blue Hackenbush is in P, but that Red-Blue-Green Hackenbush is NP-Hard, when played under mis\`ere rules.
\end{abstract}

\section{Introduction}

Hackenbush is a game that is played on a graph, with coloured edges,
that is connected to a ground defined arbitrarily before the game
begins. The rules of Red-Blue Hackenbush are as follows:

\begin{enumerate}

\item{Players take it in turn to remove edges.}
\item{Left may only remove bLue edges and Right may only remove Red edges.}
\item{Any edges not connected to the ground are also removed.}
\item{Under normal play the last player to move wins, under mis\`ere play the last player to move loses.}

\end{enumerate}

An example a of Red-Blue Hackenbush positions is given in Figures~\ref{hex1}.  The vertices that are labelled with a ``$g$'' are the vertices that are connected to the ground.

\begin{figure}[htb]
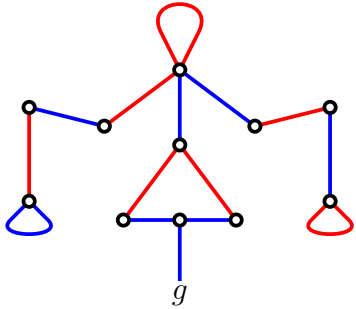

\begin{center}
\begin{graph}(4,4)

\graphnodecolour{1}\graphlinewidth{0.05}
\roundnode{a}(2,0)\roundnode{b}(2,1)\roundnode{c}(1.25,1)\roundnode{d}(2.75,1)
\roundnode{e}(2,2)\roundnode{f}(2,3)\roundnode{g}(1,2.25)\roundnode{h}(0,2.5)
\roundnode{i}(0,1.25)\roundnode{j}(3,2.25)\roundnode{k}(4,2.5)\roundnode{l}(4,1.25)

\edge{a}{b}[\graphlinecolour(0,0,1)]\edge{b}{c}[\graphlinecolour(0,0,1)]
 \edge{b}{d}[\graphlinecolour(0,0,1)] \edge{c}{e}[\graphlinecolour(1,0,0)]
\edge{d}{e}[\graphlinecolour(1,0,0)]
\edge{e}{f}[\graphlinecolour(0,0,1)]\edge{f}{g}[\graphlinecolour(1,0,0)]
\edge{g}{h}[\graphlinecolour(0,0,1)]\edge{h}{i}[\graphlinecolour(1,0,0)]
\edge{f}{j}[\graphlinecolour(0,0,1)]\edge{j}{k}[\graphlinecolour(1,0,0)] \edge{k}{l}[\graphlinecolour(0,0,1)]
\loopedge{f}(0.25,0.5)(-0.25,0.5)[\graphlinecolour(1,0,0)]
\loopedge{i}(0.25,-0.25)(-0.25,-0.25)[\graphlinecolour(0,0,1)]
\loopedge{l}(0.25,-0.25)(-0.25,-0.25)[\graphlinecolour(1,0,0)]

\nodetext{a}{$g$}

\end{graph}
\end{center}
\caption{An example of a Red-Blue Hackenbush position}\label{hex1}
\end{figure}

In Winning Ways \cite{WW}, the authors used different variants of
Hackenbush to illustrate all parts of the theory for normal play
games.  For this reason it is worth studying when considering
mis\`ere play games.  However it has been shown that determining the 
outcome of a general position of Red-Blue Hackenbush under normal play
is NP-hard (for an explanation of NP-hardness see \cite{GJ}).

In this paper we will also be using the following definition for the four possible outcome classes of a normal or mis\`ere play game;

\begin{definition}\cite{LIP}  We define the following;
\begin{itemize}
\item{$\lef=\{G|\hbox{Left wins playing first or second in } G\}$.}
\item{$\ri=\{G|\hbox{Right wins playing first or second in } G\}$.}
\item{$\pre=\{G|\hbox{The second player to move wins in } G\}$.}
\item{$\n=\{G|\hbox{The first player to move wins in } G\}$.}
\end{itemize}
\end{definition}

For further information about combinatorial game theory see \cite{LIP}, \cite{WW} or \cite{ONAG}.

\section{Red-Blue and Red-Blue-Green Mis\`ere Hackenbush}

You might expect that when we consider Red-Blue Hackenbush under
mis\`ere rules that it is still hard to determine the winner.
However it is actually very easy to determine the winner of a
Red-Blue Hackenbush position under mis\`ere rules and it can be done
in polynomial time as shown in Theorem \ref{hackenbush}.

\begin{definition}
A ``grounded'' edge, is an edge that is connected directly to the
ground.
\end{definition}

\begin{theorem}\label{hackenbush}  Let $G$ be a game of Red-Blue
mis\`ere Hackenbush, and let $B$ and $R$ be the number of grounded
blue and red edges respectively, then the outcome of $G$ can be
determined by the following formula:

$$G\in\begin{cases}
&\lef\text{, if $B>R$}\\
&\ri\text{, if $R>B$}\\
&\n\text{, if $B=R$}
\end{cases}$$

\end{theorem}

\noindent \begin{proof}

Let there be $R$ grounded Red edges and $B$ grounded Blue edges, and
consider the case where $R\geq B$.  Consider Left moving first. His
winning move will be to remove one of his own grounded edges.
Regardless of what Right does in response to this, Left can keep
removing his grounded edges.

Once Left has removed all of these edges, there will be at least one
grounded Right edge, and Left wins regardless if he is to move first
or second in this situation.

Left can only win moving second if $R>B$, since Right taking one of
his grounded Red edges will be moving to the situation $R\geq B$,
and it will be Left's turn to move.  If Right chooses not to take
one of his grounded Red edges , then again Left takes one of his
grounded Blue edges, and again wins.

The situation $B\leq R$ follows by symmetry.

\end{proof}

So this means that all we have to do to find the outcome class for a
game of Red-Blue mis\`ere Hackenbush, we simply count the number of
grounded red and blue edges and the difference will tell us the
outcome class.  This can clearly be done in polynomial time, which
means that Red-Blue mis\`ere Hackenbush is neither NP-complete or
NP-hard.

\subsection{Red-Blue-Green Mis\`ere Hackenbush}

The rules of Red-Blue-Green Hackenbush are identical to the rules of Red-Blue Hackenbush, but for one additional rule.  That is green edges, which may be removed by both players, in the diagrams that follow green edges will be represented by thick edges.  It turns out to be a far more complicated game.

\noindent PROBLEM: \textbf{RED-BLUE-GREEN MIS\`ERE HACKENBUSH}

\vskip2pt

\noindent INSTANCE: A position of Red-Blue-Green Mis\`ere Hackenbush
$G$. \vskip2pt

\noindent QUESTION: What is the outcome of $G$?

\begin{theorem}\label{mrbh}
Red-Blue-Green Mis\`ere Hackenbush is NP-hard.
\end{theorem}

\begin{proof}
To prove this we will a do a transformation from Red-Blue Hackenbush under normal play rules.
First we note two things, as previously stated, it is known that determining the outcome of a general
position of normal play Red-Blue Hackenbush is NP-hard.  It is also known that we can think of the
ground in Hackenbush as being a single vertex, which is drawn as a ground with separate vertices
for clarity in diagrams, \cite{TF}, page 40.  With this in mind we will make our transformation.

The transformation will be as follows, start with a general Red-Blue Hackenbush position $G$.  Next
take the same position and replace the ground, and all the vertices that are on the ground with a single vertex
and call this game $G'$.  Lastly attach $G'$ to a single grounded green edge, and call this game $G_m$.  This process
is illustrated in Figure \ref{transform}.  The figure shows two red and blue edges, this is simply to illustrate the process, however the graph $G$ can be any graph only if the edges are all coloured red or blue.

\begin{figure}[htb]
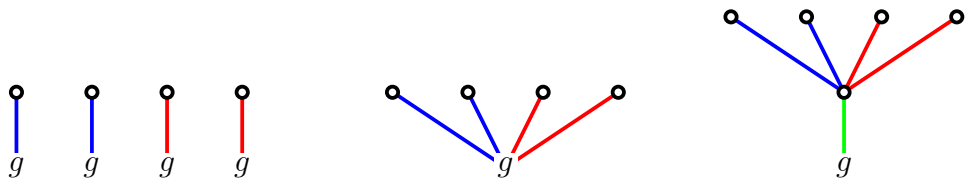

\begin{center}
\begin{graph}(12.5,2)

\graphnodecolour{1}\graphlinewidth{0.05}
\roundnode{1}(0,0)\roundnode{2}(0,1)\roundnode{3}(1,0)\roundnode{4}(1,1)
\roundnode{5}(2,0)\roundnode{6}(2,1)\roundnode{7}(3,0)\roundnode{8}(3,1)
\roundnode{9}(6.5,0)\roundnode{10}(11,0)\roundnode{11}(11,1)\roundnode{12}(7,1)
\roundnode{13}(6,1)\roundnode{14}(5,1)\roundnode{15}(8,1)\roundnode{16}(9.5,2)\roundnode{17}(10.5,2)\roundnode{18}(11.5,2)\roundnode{19}(12.5,2)

\edge{1}{2}[\graphlinecolour(0,0,1)]
\edge{3}{4}[\graphlinecolour(0,0,1)]
\edge{9}{13}[\graphlinecolour(0,0,1)]
\edge{9}{14}[\graphlinecolour(0,0,1)]\edge{11}{16}[\graphlinecolour(0,0,1)]\edge{11}{17}[\graphlinecolour(0,0,1)]

\edge{5}{6}[\graphlinecolour(1,0,0)]
\edge{7}{8}[\graphlinecolour(1,0,0)]
\edge{9}{12}[\graphlinecolour(1,0,0)]
\edge{9}{15}[\graphlinecolour(1,0,0)]
\edge{11}{18}[\graphlinecolour(1,0,0)]\edge{11}{19}[\graphlinecolour(1,0,0)]

\edge{10}{11}[\graphlinecolour(0,1,0)]

\nodetext{1}{$g$}\nodetext{3}{$g$}\nodetext{5}{$g$}\nodetext{7}{$g$}\nodetext{9}{$g$}\nodetext{10}{$g$}

\end{graph}
\end{center}
\caption{Transformation of $G$ to $G_m$.}\label{transform}
\end{figure}

If we are playing $G_m$ under mis\`ere rules, then neither player will want to cut the single green edge, since doing
so will remove every edge in the game, and thus the next player will be unable to move and therefore win under mis\`ere rules.
So both players will want to move last on the graph $G'$ that is attached to the single green edge, thus forcing your opponent
to remove the green edge, which will result in you winning the game.  In other words, whoever wins $G'$ under normal play rules, will
also win $G_m$ under mis\`ere play rules, and since determining the outcome of $G'$ is NP-hard, determing the outcome of $G_m$ is also NP-hard.
So the theorem is proven.
\end{proof}

\begin{problem}  If we restrict Red-Blue-Green mis\`ere Hackenbush
to collections of strings with only grounded green edges is it still
NP-hard to determine the outcome class?  If not what is the
solution?
\end{problem}

\section*{\normalsize Acknowledgements}

I would like to thank my supervisor Keith Edwards for his help and
advice on the theory of NP-completeness and the NP-hard proof.

\end{document}